\documentclass[11pt]{article}
\usepackage{amsmath,amssymb,textcomp,stmaryrd,xifthen,psfrag,graphicx,color}
\usepackage{amsfonts,amsthm}
\usepackage{color}

\SetSymbolFont{stmry}{bold}{U}{stmry}{m}{n} 
\usepackage[T1]{fontenc}
\date{}

\newtheorem{theorem}{Theorem}
\newtheorem{lemma}[theorem]{Lemma}
\newtheorem{cor}[theorem]{Corollary}

\theoremstyle{definition} 
\newtheorem{remark}[theorem]{Remark}

\newcommand{\dual}[2]{\langle#1\hspace*{.5mm},#2\rangle}
\newcommand{\vdual}[2]{(#1\hspace*{.5mm},#2)}
\newcommand{\abs}[1]{\vert #1 \vert}
\newcommand{\norm}[3][]{#1\|#2#1\|_{#3}}
\newcommand{\snorm}[2]{|#1|_{#2}}

\newcommand{\diam}{\mathrm{diam}}
\newcommand{\wilde}{\widetilde}

\newcommand{\trace}{\gamma}

\def\eps{\varepsilon}

\newcommand{\R}{\ensuremath{\mathbb{R}}}
\newcommand{\N}{\ensuremath{\mathbb{N}}}

\newcommand{\cH}{\ensuremath{\widetilde H}}

\newcommand{\nn}{\ensuremath{\mathbf{n}}}

\newcommand{\Tt}{\ensuremath{\mathcal{T}}}
\newcommand{\Ss}{\ensuremath{\mathcal{S}}}

\newcommand{\OO}{\ensuremath{\mathcal{O}}}

\newcommand{\Ee}{\ensuremath{\mathcal{E}}}

\title{A finite element method for\\elliptic Dirichlet boundary control problems
\thanks{Supported by CONICYT through FONDECYT project 1170672.}}
\author{Michael Karkulik\thanks{Departamento de Matem\'atica, Universidad T\'ecnica Federico Santa Mar\'ia,
  Avenida Espa\~na 1680, Valpara\'iso, Chile,
  \texttt{michael.karkulik@usm.cl}}}
\begin{document}
\maketitle
\begin{abstract}
  We consider the finite element discretization of an optimal Dirichlet boundary control problem for the Laplacian,
  where the control is considered in $H^{1/2}(\Gamma)$. To avoid computing the latter norm numerically,
  we realize it using the $H^1(\Omega)$ norm of the harmonic extension of the control.
  We propose a mixed finite element discretization, where the harmonicity of the solution is
  included by a Lagrangian multiplier. In the case of convex polygonal domains,
  optimal error estimates in the $H^1$ and $L_2$ norm are proven.
  We also consider and analyze the case of control constrained problems.

\end{abstract}
\section{Introduction}
In this work, we consider elliptic Dirichlet boundary control problems.
For $\Omega\subset\R^d$ an open Lipschitz domain with boundary $\Gamma:=\partial\Omega$,
given functions $\wilde u_d\in L_2(\Omega)$ (the \textit{desired state}) and $f\in H^{-1}(\Omega)$,
and a parameter $\lambda>0$, we aim to find a minimizer $g$ (the~\textit{control}) of
\begin{subequations}\label{eq:problem}
\begin{align}\label{eq:problem:min}
  \frac{1}{2}\norm{\widetilde u-\widetilde u_d}{L_2(\Omega)}^2 + \frac{\lambda}{2}\norm{g}{H^{1/2}(\Gamma)}^2,
\end{align}
where $\widetilde u$ and $g$ are related via the state equation
\begin{align}\label{eq:problem:pde}
  \begin{split}
    -\Delta \widetilde u&= f \quad\text{ in } \Omega,\\
    \widetilde u &= g \quad\text{ on } \Gamma,
  \end{split}
\end{align}
\end{subequations}
with possible control constraints $a\leq g \leq b$ almost everywhere on $\Gamma$.
We emphasize here that without loss of generality, the function $f$ can be assumed to be $0$.
While the importance of optimal control of PDE is widely recognized as witnessed by the
monographs~\cite{Lions_71,troltzsch_10}, the numerical approximation of such problems
has also been subject to active research in the last decades, cf.~\cite{delosReyes_15}.
Quite naturally, finite element methods are used in order to approximate the involved PDE.
By now, there is vast literature on finite element approximation of distributed control problems,
i.e., where the control acts on (part of) the domain $\Omega$.
Neumann boundary control problems, where instead of equation~\eqref{eq:problem:pde}
the control enters via a Neumann boundary condition, have been 
among the first PDE optimal control problems which were analyzed numerically,
cf.~\cite{Geveci_Neumann_79}.
Dirichlet boundary control, as considered in the present work, has applications in computational fluid dynamics,
cf.~\cite{GunzburgerHS_92,FursikovGH_98} and the references therein,
as well as the work~\cite{JohnW_DirichletControl_09} dealing with the optimal control
of an aircrafts lift and drag. 
Coming back to the finite element approximation of the model problem~\eqref{eq:problem}
a widely used definition for the $H^{1/2}(\Gamma)$-norm is given by
\begin{align*}
  \norm{g}{H^{1/2}(\Gamma)}^2 := \norm{g}{L_2(\Gamma)}^2
  + \int_\Gamma\int_\Gamma \frac{\abs{g(x)-g(y)}^2}{\abs{x-y}^2}\,ds_xds_y.
\end{align*}
The double integral in this definition is difficult to realize numerically due
to its singular kernel and its nonlocality. To circumvent this problem,
three approaches can be found in the literature.
The first one is to use a singularly perturbed Robin boundary condition
of the form $\varepsilon \partial_\nn u + u = g$
with small $\varepsilon$ instead of the Dirichlet boundary condition,
cf.~\cite{BBFM_03,CasasMR_09,ChangGY_17}.
The second approach is to use the $L_2(\Gamma)$ norm, that is, instead
of~\eqref{eq:problem:min}, the cost functional
\begin{align}\label{eq:problem:uw}
  \frac{1}{2}\norm{\widetilde u-\widetilde u_d}{L_2(\Omega)}^2 + \frac{\lambda}{2}\norm{g}{L_2(\Gamma)}^2,
\end{align}
is minimized, cf. the discussion in~\cite{KunischV_07}. Finite element methods and corresponding error estimates
for the minimization problem~\eqref{eq:problem:uw} can be
found in~\cite{CasasR_06,DeckelnickGH_09}, and recently in~\cite{MayRV_13}.
It is clear that these finite element methods can not be based on the standard variational formulation
of the equation~\eqref{eq:problem:pde}. The reason for this is that the corresponding
energy space $H^1(\Omega)$ for $\wilde u$ does not allow for a surjective trace operator
with image $L_2(\Gamma)$. The remedy, as carried out in~\cite{MayRV_13}, is a so-called \textit{ultra weak}
formulation of the state equation,  which is obtained by integrating by parts two times the
equation~\eqref{eq:problem:pde}.
This loosens the regularity assumptions on $\wilde u$ to be mainly $L_2(\Omega)$,
but imposes stronger regularity $H^1_0(\Omega)\cap H^2(\Omega)$ on the test side.
If $\Omega$ is smooth or a convex polygon in $\R^2$,
the implementation of $H^2(\Omega)$-conforming finite element spaces can be avoided
using a careful regularity theory, cf.~\cite{ApelMP_15,ApelNP_16},
and a discrete optimality system based on standard finite element spaces can be derived.
Variations of this approach accounting for mixed finite element methods~\cite{GongY_Mixed_11} or
symmetric interior penalty Galerkin methods~\cite{BennerY_17} do exist.
However, as the minimization is still carried out using the $L_2(\Gamma)$ norm,
an a priori error estimate of at most $\norm{\wilde u-\wilde u_h}{L_2(\Omega)}=\OO(h^{3/2})$
can be obtained in the case of linear finite elements, regardless of the regularity of $\wilde u$.
This brings us to the third approach,  which is called \textit{energy space approach} in the literature.
In order to loosen the restriction to smooth domains or convex polygons and to
obtain the optimal a priori error estimate $\norm{\wilde u-\wilde u_h}{L_2(\Omega)} = \OO(h^2)$,
it is inevitable to minimize~\eqref{eq:problem:min}, and hence
a first task is to realize the norm $\norm{g}{H^{1/2}(\Gamma)}$ in a
way which is more convenient for numerical purposes.
The pioneering work considering this very approach is~\cite{OfPS_15}. There, the authors
consider minimization with respect to the semi-norm
\begin{align}\label{eq:seminorm:harmext}
  \snorm{g}{H^{1/2}(\Gamma)} := \norm{\nabla \Ee g}{L_2(\Omega)}
\end{align}
with $\Ee$ being the harmonic extension operator. The right-hand side of the above definition
is then integrated by parts and converts into an integral over $\Gamma$ involving the trace and the normal
derivative of $\Ee g$. The normal derivative is realized numerically by involving the
Dirichlet-to-Neumann map (a boundary integral operator also known as \textit{Steklov-Poincare operator}).
In order to facilitate the derivation of a discrete method, the 
the authors of the recent work~\cite{ChowdhuryGN_17} aim at a setting which is completely posed in the domain $\Omega$
and consider therefore the minimization of
\begin{align}\label{eq:chow}
  \norm{u-u_d}{L_2(\Omega)}^2 + \lambda\norm{\nabla g}{L^2(\Omega)}^2,
\end{align}
for $u,g\in H^1(\Omega)$ under the PDE constraint
\begin{align*}
  -\Delta u&= 0 \quad\text{ in } \Omega,\\
  u &= \textrm{trace}(g) \quad\text{ on } \Gamma.
\end{align*}
It is clear that a function $g$ minimizing~\eqref{eq:chow} is harmonic. This provides the link
between the continuous formulations in the works~\cite{OfPS_15,ChowdhuryGN_17}, and furthermore
it shows that $g=u$, which makes numerical methods considerably cheaper as $g$ and $u$ do not have to be
approximated separately.
As a consequence of working in energy spaces, in both works the authors obtain the optimal error estimates
$\norm{\wilde u-\wilde u_h}{L_2(\Omega)} = \OO(h^2)$ and 
$\norm{\wilde u-\wilde u_h}{H^1(\Omega)} = \OO(h)$ for piecewise linear finite elements
in the case of $\Omega$ being a convex polygonal domain.
In~\cite{ChowdhuryGN_17}, the authors also provide a-posteriori error estimates, and
the work~\cite{GongLTY_18} shows convergence of a corresponding adaptive algorithm.
At this point, we mention the following. In~\cite{OfPS_15,ChowdhuryGN_17}
minimization is done with respect to a \textit{seminorm}.
As seminorms evaluate to $0$ for constants $c\in\R$,
a property of these methods is the simple fact that \textit{constant controls are for free}.
In the work~\cite{OfPS_15}, the use of the seminorm is essential in order to obtain a boundary integral via
integration by parts. On the other hand, regarding the work~\cite{ChowdhuryGN_17},
if we would use the full norm $\norm{g}{H^1(\Omega)}$ instead of
the seminorm $\norm{\nabla g}{L^2(\Omega)}$ in~\eqref{eq:chow}, then a minimizer $g$ is not necessarily harmonic.
Hence, $u$ and $g$ have to be approximated separately, and this considerably increases the cost of a numerical method.
Therefore, the purpose of the work at hand is to reconsider the derivation of a simple finite element method
for the problem~\eqref{eq:problem}, and furthermore, to provide analysis in the case of control constraints
$a \leq g \leq b$ a.e. on $\Gamma$.
Continuing the ideas from~\cite{OfPS_15,ChowdhuryGN_17}, our approach is based on the observation that
$\norm{g}{H^{1/2}(\Gamma)} \simeq \norm{\Ee g}{H^1(\Omega)}$, such that we will realize
the norm $\norm{g}{H^{1/2}(\Gamma)}$ as $\norm{\Ee g}{H^1(\Omega)}$ in our cost functional.
Another way to put this is to minimize
\begin{align*}
  \norm{u-u_d}{L_2(\Omega)}^2 + \lambda\norm{u}{H^1(\Omega)}^2,
\end{align*}
under the PDE constraint $-\Delta u =0$ and possible constraints $a\leq \rm{trace}(u)\leq b$.
We conclude that the first-order optimality condition of this functional
yields a variational formulation posed on the space of harmonic functions.
In order to obtain a variational formulation which can be easily discretized,
we include the harmonicity of the trial and test spaces via Lagrangian multipliers.
This yields a mixed finite element formulation which can be discretized by simple finite element spaces.
\section{Problem statement and main results}
\subsection{Notation}
Let $\Omega\subset\R^d$ be a bounded domain with a Lipschitz boundary $\Gamma$.
We use the classical Lebesgue and Sobolev spaces $L_2(\Omega)$ and $H^1(\Omega)$,
as well as the space $H^{1/2}(\Gamma)$ on the boundary, which is defined via
the Sobolev-Slobodeckii norm.
The space $H^{-1/2}(\Gamma)$ is the topological dual of $H^{1/2}(\Gamma)$
with respect to the extended $L_2(\Gamma)$ scalar product.
We will frequently use the $L_2(\Omega)$ scalar product, which we denote
by $\vdual{\cdot}{\cdot}_\Omega$.
The operation of taking the trace on $\Gamma$ is denoted by $\gamma$; it is well known that
$\gamma:H^1(\Omega)\rightarrow H^{1/2}(\Gamma)$ is continuous.
Set $H^1_0(\Omega) = \left\{ u\in H^1(\Omega)\mid \gamma u = 0 \right\}$.
An argument based on the Rellich compactness theorem shows that an equivalent
norm on $H^1(\Omega)$ is given by
\begin{align}\label{eq:H1norm}
  \norm{\nabla v}{L_2(\Omega)} + \norm{\gamma v}{H^{1/2}(\Gamma)}
  \simeq
  \norm{v}{H^1(\Omega)}.
\end{align}
For a function $q\in H^{1/2}(\Gamma)$, we denote by $\Ee q\in H^1(\Omega)$ the harmonic extension,
that is, $\Ee q$ is the unique function such that
\begin{align}\label{eq:harmext}
  \begin{split}
    \vdual{\nabla\Ee q}{\nabla v}_\Omega &= 0 \quad\text{ for all } v\in H^1_0(\Omega),\\
    \gamma\Ee q &= q.
  \end{split}
\end{align}
We define the space of weakly harmonic functions by
\begin{align}\label{eq:harmfun}
  \cH^1(\Omega) := \Ee H^{1/2}(\Gamma) = 
  \left\{ u\in H^1(\Omega) \mid \vdual{\nabla u}{\nabla v}_\Omega = 0
    \quad\text{ for all }v\in H^1_0(\Omega) \right\}.
\end{align}
We will use the well-known fact that for $u\in H^1(\Omega)$, it holds
\begin{align}\label{eq:harm:best}
  \norm{\nabla(u-\Ee \gamma u)}{L_2(\Omega)} = \inf_{\psi\in\cH^1(\Omega)}\norm{\nabla(u-\psi)}{L_2(\Omega)}.
\end{align}
For the analysis of the constrained problem, we introduce the closed convex sets $H^1_{\rm ad}(\Omega)$
and $\cH^1_{\rm ad}(\Omega)$, which consist of functions $v$ in the respective spaces such that
the trace $\trace v$ satisfies the constraint~\eqref{eq:problem:new:box}.
\subsection{The continuous minimization problem}
First, we will rewrite problem~\eqref{eq:problem}. To that end,
denote by $u_f\in H^1_0(\Omega)$ the weak solution of the homogeneous Dirichlet problem
$-\Delta u_f=f$. If we set $u_d:= \wilde u_d-u_f$, then
problem~\eqref{eq:problem} is equivalent to the minimization of
\begin{subequations}\label{eq:problem:hom}
\begin{align}\label{eq:problem:hom:min}
  \frac{1}{2}\norm{u-u_d}{L_2(\Omega)}^2 + \frac{\lambda}{2}\norm{g}{H^{1/2}(\Gamma)}^2,
\end{align}
under the constraint
\begin{align}\label{eq:problem:hom:pde}
  \begin{split}
    -\Delta u&= 0 \quad\text{ in } \Omega,\\
    u &= g \quad\text{ on } \Gamma,
  \end{split}
\end{align}
supplemented with possible constraints
\begin{align}\label{eq:problem:hom:con}
  a \leq g \leq b \quad\text{ almost everywhere on } \Gamma.
\end{align}
\end{subequations}
It is known that $\norm{u}{H^1(\Omega)} \simeq \norm{g}{H^{1/2}(\Omega)}$ for the solution
$u\in H^1(\Omega)$ of~\eqref{eq:problem:hom:pde}. The $\lesssim$ in the preceding inequality
follows from stability of the problem~\eqref{eq:problem:hom:pde}, while $\gtrsim$ is the trace
theorem. Hence, up to a multiplicative constant, the minimization of~\eqref{eq:problem:hom:min}
under the constraint~\eqref{eq:problem:hom:pde} is equivalent to the minimization of
\begin{subequations}\label{eq:problem:new}
\begin{align}
  \frac{1}{2}\norm{u-u_d}{L_2(\Omega)}^2 + \frac{\lambda}{2}\norm{u}{H^1(\Omega)}^2.
  \label{eq:problem:new:min}
\end{align}
under the constraint
\begin{align}\label{eq:problem:new:pde}
  \begin{split}
    -\Delta u&= 0 \quad\text{ in } \Omega,\\
    u &= g \quad\text{ on } \Gamma.
  \end{split}
\end{align}
and, if applicable,
\begin{align}\label{eq:problem:new:box}
  a \leq g \leq b \quad\text{ almost everywhere on } \Gamma.
\end{align}
\end{subequations}
Note that the space of harmonic functions is closed in $H^1(\Omega)$.
Hence, standard considerations as presented in~\cite{GlowinskiLT_81,troltzsch_10}
show that the unconstrained problem~\eqref{eq:problem:new:min}-\eqref{eq:problem:new:pde}
as well as the constrained problem~\eqref{eq:problem:new} are well posed.
This follows in particular from Lemma~\ref{lem:new:wellposed} below.
\begin{lemma}\label{thm:prob:new:wellposed}
  Problem~\eqref{eq:problem:new:min}-\eqref{eq:problem:new:pde} is well-posed:
  There exists a unique function $u\in H^1(\Omega)$ such
  that $u$ is harmonic, has trace $g$, and minimizes~\eqref{eq:problem:new:min}.
  There holds stability
  \begin{align*}
    \norm{u}{H^1(\Omega)} \lesssim \norm{u_d}{L_2(\Omega)}.
  \end{align*}
  The same is true for the constrained problem~\eqref{eq:problem:new}.
  \qed
\end{lemma}
\subsection{Finite element discretization of the unconstrained problem}
Our finite element discretization of the
unconstrained problem~\eqref{eq:problem:new:min}-\eqref{eq:problem:new:pde} is to find
$(u_h,\varphi_h)\in\Ss^1(\Tt_h)\times\Ss^1_0(\Tt_h)$ such that
\begin{align}\label{eq:var:disc}
  \begin{split}
  (\lambda+1)\vdual{u_h}{v_h}_\Omega + \lambda\vdual{\nabla u_h}{\nabla v_h}_\Omega 
  + \vdual{\nabla \varphi_h}{\nabla v_h}_\Omega &= \vdual{u_d}{v_h}\\
  \vdual{\nabla u_h}{\nabla \psi_h}_\Omega &= 0
  \end{split}
\end{align}
for all $(v_h,\psi_h)\in\Ss^1(\Tt_h)\times\Ss^1_0(\Tt_h)$.
Here, $\Ss^1(\Tt_h)$ is the space of piecewise affine, globally continuous functions
on some partition $\Tt_h$ of $\Omega$, and $\Ss^1_0(\Tt_h)$ is the same
space equipped with vanishing boundary conditions, cf. Section~\ref{sec:disc} below for a
precise definition. Our main result is the following.
\begin{theorem}\label{thm:est}
  There exists a unique solution $(u_h,\varphi_h)\in\Ss^1(\Tt_h)\times\Ss^1_0(\Tt_h)$
  of~\eqref{eq:var:disc}. Furthermore, if $u\in H^1(\Omega)$ is the exact solution
  of~\eqref{eq:problem:new:min}-\eqref{eq:problem:new:pde}, then there holds
  the quasi optimality
  \begin{align*}
    \norm{u-u_h}{H^1(\Omega)} \lesssim \inf_{w_h\in\Ss^1(\Tt_h)}\norm{u-w_h}{H^1(\Omega)}
    + \sup_{\norm{f}{L_2(\Omega)}=1} \inf_{v_h\in \Ss^1_0(\Tt_h)} \norm{v_f-v_h}{H^1(\Omega)},
  \end{align*}
  where $v_f\in H^1_0(\Omega)$ is the solution of $-\Delta v_f=f$ and $v_f = 0$ on $\Gamma$.
\end{theorem}
As a corollary, we obtain the following convergence rates.
\begin{cor}\label{cor:est}
  Suppose that $\Omega\in\R^2$ is a convex polygon.
  Let $u$ be the solution of problem~\eqref{eq:problem:new:min}-\eqref{eq:problem:new:pde},
  and $u_h\in\Ss^1(\Tt_h)$ the finite element approximation of $u$ defined
  in~\eqref{eq:var:disc}. Then it holds
  \begin{align*}
    \norm{u-u_h}{H^1(\Omega)} = \OO(h),
    \quad\text{ and }\quad \norm{u-u_h}{L_2(\Omega)} = \OO(h^2),
  \end{align*}
  as well as
  \begin{align*}
    \norm{u-u_h}{H^{1/2}(\Gamma)} = \OO(h),
    \quad\text{ and } \norm{u-u_h}{L_2(\Gamma)} = \OO(h^{3/2}).
  \end{align*}
  where $h$ is the maximum mesh size of $\Tt_h$.
  \qed
\end{cor}
\subsection{Finite element discretization of the constrained problem}
Our finite element discretization of the
constrained problem~\eqref{eq:problem:new} is to find
$(u_h,\varphi_h)\in\Ss^1_{\rm ad}(\Tt_h)\times\Ss^1_0(\Tt_h)$ such that
\begin{align}\label{eq:var:constrained:disc}
  \begin{split}
  (\lambda+1)\vdual{u_h}{v_h-u_h}_\Omega + \lambda\vdual{\nabla u_h}{\nabla (v_h-u_h)}_\Omega 
  + \vdual{\nabla \varphi_h}{\nabla (v_h-u_h)}_\Omega &\geq \vdual{u_d}{v_h-u_h}\\
  \vdual{\nabla u_h}{\nabla \psi_h}_\Omega &= 0
  \end{split}
\end{align}
for all $(v_h,\psi_h)\in\Ss^1_{\rm ad}(\Tt_h)\times\Ss^1_0(\Tt_h)$.
Here, the convex set $\Ss^1_{\rm ad}(\Tt_h)$ denotes the subset of functions in $\Ss^1(\Tt_h)$
which fulfill the bound~\eqref{eq:problem:new:box}.
Our main result in this case is the following best approximation result.
\begin{theorem}\label{thm:cea:constrained}
  There exists a unique solution $(u_h,\varphi_h)\in\Ss^1_{\rm ad}(\Tt_h)\times\Ss^1_0(\Tt_h)$
  of~\eqref{eq:var:constrained:disc}.
  Furthermore, if $\Omega\in\R^2$ is a convex polygon and
  the exact solution $u$ of~\eqref{eq:problem:new} fulfills
  $u\in H^2(\Omega)$, then, for every $\eps>0$ there exists a constant $C_\eps>0$ such that
  \begin{align*}
    \norm{u-u_h}{H^1(\Omega)}
    \leq C_\eps
    \inf_{w_h\in \wilde\Ss^1_{\rm ad}(\Tt_h)}
    \left(
    \norm{\nabla (u-w_h)}{L_2(\Omega)}
    + \norm{u-w_h}{H^{-1/2+\eps}(\Gamma)}^{1/2}
    \right)\\
    + \sup_{\norm{f}{L_2(\Omega)}=1} \inf_{v_h\in \Ss^1_0(\Tt_h)} \norm{v_f-v_h}{H^1(\Omega)}
  \end{align*}
  \qed
\end{theorem}
\begin{remark}
  In the constrained case it is not clear how to obtain convergence rates as in Corollary~\ref{cor:est},
  as it is not obvious how to construct $w_h\in\wilde\Ss^1_{\rm ad}(\Tt_h)$ such that the right-hand side
  of the estimate in Theorem~\ref{thm:cea:constrained} converges with $\OO(h)$.
  This difficulty is due to the constraints. In fact, if we want to bound
  \begin{align}\label{eq:remark}
    \inf_{w_h\in \wilde\Ss^1(\Tt_h)}
    \left(
    \norm{u-w_h}{H^1(\Omega)}
    + \norm{u-w_h}{H^{-1/2}(\Gamma)}^{1/2}
    \right)
  \end{align}
  we choose an arbitrary $\wilde w_h\in\wilde\Ss^1(\Tt_h)$ and set
  $w_h := \wilde w_h + \Ee_h I_h^{\rm Car}\trace (u-\wilde w_h)$. Here, $I_h^{\rm Car}$ is the so-called
  Carstensen interpolant~\cite{Carstensen_99}. This operator is bounded in $H^{1/2}(\Gamma)$, such that
  \begin{align*}
    \norm{u-w_h}{H^1(\Omega)} \lesssim \norm{u-\wilde w_h}{H^1(\Omega)}
    + \norm{I_h^{\rm Car}\trace (u-\wilde w_h)}{H^{1/2}(\Gamma)}
    \lesssim \norm{u-\wilde w_h}{H^1(\Omega)}.
  \end{align*}
  Furthermore, the approximation error of $I_h^{\rm Car}$ has vanishing local average on nodal patches.
  This crucial property can be used to prove that $\norm{u - I_h^{\rm Car}u}{H^{-1/2}(\Gamma)}\lesssim h \norm{u}{H^{1/2}(\Gamma)}$,
  and hence
  \begin{align}\label{eq:Hminus12}
    \norm{u-w_h}{H^{-1/2}(\Gamma)}^{1/2} = \norm{u-\wilde w_h-I_h^{\rm Car}\trace(u-\wilde w_h)}{H^{-1/2}(\Gamma)}^{1/2}
    \lesssim h^{1/2} \norm{u-\wilde w_h}{H^{1}(\Omega)}^{1/2}
    \lesssim \norm{u-\wilde w_h}{H^1(\Omega)} + h
  \end{align}
  Thus, the infimum~\eqref{eq:remark} can be bounded by
  \begin{align*}
    \inf_{w_h\in \wilde\Ss^1(\Tt_h)} \norm{u-\wilde w_h}{H^1(\Omega)} + h,
  \end{align*}
  which, in turn, can be bounded by Lemma~\ref{lem:cea:harmonic} and ubiquitous approximation results in Sobolev
  spaces by $\OO(h)$. Quasi-interpolation operators preserving constraints do exist, cf.~\cite{NochettoW_02},
  but it is not clear how to tweak them to obtain the bound~\eqref{eq:Hminus12}.
\end{remark}
\section{Analysis of the continuous problem}
\subsection{Variational formulation of the continuous minimization problem}
We will now present the different variational formulations of problem~\eqref{eq:problem:new}
and comment on their equivalence. Let $S:H^{1/2}(\Gamma)\rightarrow H^1(\Omega)$
denote the solution operator of~\eqref{eq:problem:new:pde}, i.e., $S(g)= u$ where
\begin{align*}
  \vdual{\nabla u}{\nabla v}_\Omega &= 0 \quad\text{ for all } v\in H^1_0(\Omega),\\
  u&=g \quad\text{ on } \Gamma.
\end{align*}
It follows by well-known arguments that the unconstrained
problem~\eqref{eq:problem:new:min}-\eqref{eq:problem:new:pde}, for example,
is equivalent to the problem to find $g\in H^{1/2}(\Gamma)$ such that
\begin{align}\label{eq:foc}
  \left( \lambda + 1 \right)\vdual{Sg}{St}_{\Omega} + \lambda\vdual{\nabla Sg}{\nabla St}_{\Omega}
  = \vdual{u_d}{St}_\Omega\quad\text{ for all } t\in H^{1/2}(\Gamma),
\end{align}
or, put differently, to find $u\in \cH^1(\Omega)$ such that
\begin{align}\label{eq:foc:2}
  (\lambda+1)\vdual{u}{v}_\Omega + \lambda\vdual{\nabla u}{\nabla v}_\Omega = \vdual{u_d}{v}_\Omega
  \quad\text{ for all } v\in \cH^1(\Omega).
\end{align}
The equivalence between~\eqref{eq:foc} and~\eqref{eq:foc:2} is given by $g = \trace u$.
In order to avoid the discretization of $S$ for the computation of $St$ in~\eqref{eq:foc}
(i.e., to fulfill the restriction $v\in\cH^1(\Omega)$ in~\eqref{eq:foc:2}), one usually introduces
the adjoint operator $S^\star$. Here, we propose to include the condition $\vdual{\nabla u}{\nabla v}_\Omega = 0$
with a Lagrangian multiplier, that is, find $(u,\varphi)\in H^1(\Omega)\times H^1_0(\Omega)$
such that
\begin{align}\label{eq:var}
  \begin{split}
  (\lambda+1)\vdual{u}{v}_\Omega + \lambda\vdual{\nabla u}{\nabla v}_\Omega 
  + \vdual{\nabla \varphi}{\nabla v}_\Omega &= \vdual{u_d}{v}\\
  \vdual{\nabla u}{\nabla \psi}_\Omega &= 0
  \end{split}
\end{align}
for all $(v,\psi)\in H^1(\Omega)\times H^1_0(\Omega)$.
The finite element method~\eqref{eq:var:disc} is a discretization
of~\eqref{eq:var} by standard finite element spaces.
We note that this problem is of saddle point type and a-priori analysis would have
to be done accordingly. For example, we would have to check
that discretizations allow for discrete inf-sup conditions.
Furthermore, the variational formulation~\eqref{eq:var:constrained} that we propose for the constrained problem
below consists of a variational inequality and equality.
In order to facilitate the a-priori analysis, we will carry it out using
the equivalent formulation~\eqref{eq:foc:2} in case of the unconstrained problem, and
the respective variational inequality for the constrained problem.
To that end, we will show that conforming discretizations of~\eqref{eq:var} are equivalent
to nonconforming discretizations of~\eqref{eq:foc:2}
and our a-priori analysis will be then carried out in the nonconforming setting only, using the Strang lemma.
For the constrained problem~\eqref{eq:problem:new:min}-\eqref{eq:problem:new:pde}-\eqref{eq:problem:new:box}
we propose the variational formulation
\begin{align}\label{eq:var:constrained}
  \begin{split}
  (\lambda+1)\vdual{u}{v-u}_\Omega + \lambda\vdual{\nabla u}{\nabla v-u}_\Omega 
  + \vdual{\nabla \varphi}{\nabla v-u}_\Omega &\geq \vdual{u_d}{v-u}\\
  \vdual{\nabla u}{\nabla \psi}_\Omega &= 0
  \end{split}
\end{align}
for all $(v,\psi)\in H^1_{\rm ad}(\Omega)\times H^1_0(\Omega)$.

Define the bilinear form
$b:H^1(\Omega)\times H^1(\Omega)\rightarrow\R$ by
\begin{align*}
  b(u,w) := \left( \lambda + 1 \right)\vdual{u}{v}_{\Omega} + \lambda\vdual{\nabla u}{\nabla v}_{\Omega}.
\end{align*}
Given $g\in L_2(\Omega)$, the problem to 
\begin{align}\label{eq:var:general}
  \text{find } u \in H^1(\Omega) \text{ such that }
  b(u,v) = \vdual{g}{v}_\Omega \quad\text{ for all } v\in H^1(\Omega)
\end{align}
is well-posed, and so is the problem to
\begin{align}\label{eq:var:constrained:general}
  \text{find } u \in H^1_{\rm ad}(\Omega) \text{ such that }
  b(u,v-u) \geq \vdual{g}{v-u}_\Omega \quad\text{ for all } v\in H^1_{\rm ad}(\Omega),
\end{align}
where $H^1_{\rm ad}(\Omega)$ denotes a closed convex subset of $H^1(\Omega)$.
For problem~\eqref{eq:var:general}, this follows immediately with the Lemma of Lax Milgram,
as $b$ is obviously elliptic and bounded on $H^1(\Omega)$,
and the linear functional $\ell(v):=\vdual{g}{v}_\Omega$ is bounded on $H^1(\Omega)$.
For problem~\eqref{eq:var:constrained:general}, this follows by the theory of variational inequalities,
cf.~\cite[Chapter~2.1]{GlowinskiLT_81}.
Note that $\cH^1(\Omega)$ is a closed subspace of $H^1(\Omega)$, and
$\cH^1_{\rm ad}(\Omega):=H^1_{\rm ad}(\Omega)\cap \cH^1(\Omega)$
is a closed and convex subset of $H^1(\Omega)$ and $\cH^1(\Omega)$.
Hence, we immediately obtain the following result.
\begin{lemma}\label{lem:new:wellposed}
  Suppose that $g\in L_2(\Omega)$. Then, the problems to
  \begin{align}\label{eq:problem:var}
    \text{find } \phi\in \cH^1(\Omega) \text{ such that } b(\phi,\psi)=\vdual{g}{\psi}_\Omega
    \quad\text{ for all } \psi\in \cH^1(\Omega)
  \end{align}
  and to
  \begin{align}\label{eq:problem:var:constrained}
    \text{find } \phi\in \cH^1_{\rm ad}(\Omega) \text{ such that } 
    b(\phi,\psi-\phi)\geq\vdual{g}{\psi-\phi}_\Omega \quad\text{ for all } \psi\in \cH^1_{\rm ad}(\Omega)
  \end{align}
  are well-posed: to each one there exists a unique solution $\phi\in \cH^1(\Omega)$, and it holds
  \begin{align}\label{lem:new:wellposed:eq1}
    \norm{\phi}{H^1(\Omega)} \lesssim \norm{g}{L_2(\Omega)}.
  \end{align}
  \qed
\end{lemma}
As already stated at the beginning of this chapter,
for $g:=u_d$ the problem~\eqref{eq:problem:new} is equivalent to
problem~\eqref{eq:problem:var}.
Likewise, by well-known arguments carried out in~\cite[Chapter~1]{GlowinskiLT_81},
problem~\eqref{eq:problem:var:constrained} is equivalent to
the constrained problem~\eqref{eq:problem:new:min}-\eqref{eq:problem:new:pde}-\eqref{eq:problem:new:box}.
Lemma~\ref{lem:new:wellposed} is stated
as to facilitate some duality arguments below, where it
is necessary to solve problem~\eqref{eq:problem:var} for general $g\in L_2(\Omega)$.
\begin{cor}\label{cor:cont:equiv}
  Problem~\eqref{eq:var}, respectively problem~\eqref{eq:var:constrained}, is well-posed, i.e.,
  there exists a unique solution $(u,\varphi)\in H^1(\Omega)\times H^1_0(\Omega)$,
  respectively $(u,\varphi)\in H^1_{\rm ad}(\Omega)\times H^1_0(\Omega)$, and
  \begin{align*}
    \norm{u}{H^1(\Omega)} + \norm{\varphi}{H^1(\Omega)}
    \lesssim \norm{u_d}{L_2(\Omega)}.
  \end{align*}
  Furthermore, problems~\eqref{eq:var} and~\eqref{eq:problem:var} with $g:=u_d$ are equivalent,
  and so are problems~\eqref{eq:var:constrained} and~\eqref{eq:problem:var:constrained}.
\end{cor}
\begin{proof}
  We start with problem~\eqref{eq:var}.
  Set $g:=u_d$ in~\eqref{eq:problem:var} and let $u\in \cH^1(\Omega)$ be the unique solution.
  Now, define $\varphi\in H^1_0(\Omega)$ as the unique solution to the problem
  \begin{align*}
    \vdual{\nabla \varphi}{\nabla\phi}_\Omega
    = \vdual{u_d}{\phi}_\Omega - (\lambda+1)\vdual{u}{\phi}_\Omega - \lambda \vdual{\nabla u}{\nabla \phi}
    \text{ for all }\phi\in H^1_0(\Omega).
  \end{align*}
  Then, it is easily seen that the pair $(u,\varphi)$ solves~\eqref{eq:var}, and fulfills the stated
  stability estimate. It suffices to show that this solution is unique. To that end,
  let $(u,\varphi)$ be a solution of~\eqref{eq:var} with vanishing right-hand side.
  It follows that $u\in\cH^1(\Omega)$ solves~\eqref{eq:problem:var} with vanishing right-hand side,
  and hence $u=0$. The resulting identity from~\eqref{eq:var} is $\vdual{\nabla \varphi}{\nabla v}_{\Omega}=0$
  for all $v\in H^1(\Omega)$, which implies $\varphi=0$.\\

  To consider problem~\eqref{eq:var:constrained} we can argue as above.
  Let $u\in\cH^1_{\rm ad}(\Omega)$ be the unique solution of
  problem~\eqref{eq:problem:var:constrained} with $g:=u_d$. Defining $\varphi\in H^1_0(\Omega)$ as above, we
  see that $(u,\varphi)$ solves~\eqref{eq:var:constrained} and fulfills the desired stability estimate.
  To show uniqueness, suppose that
  $(u_1,\varphi_1)$ and $(u_2,\varphi_2)$ are solutions of~\eqref{eq:var:constrained}. It follows that
  both $u_1$ and $u_2$ solve~\eqref{eq:problem:var:constrained}, with $g=u_d$, hence $u_1=u_2=:u$.
  It follows that for $j=1,2$
  \begin{align*}
    \vdual{\nabla\varphi_j}{\nabla \left( v-u \right)}_\Omega
    \geq \vdual{u_d}{v-u}_\Omega - \left( \lambda+1 \right) \vdual{u}{v-u}_\Omega
    - \lambda\vdual{\nabla u}{\nabla(v-u)}_\Omega
  \end{align*}
  for all $v\in H^1_{\rm ad}(\Omega)$. Let us abbreviate the right-hand side above by $\ell$, i.e.,
  \begin{align*}
    \ell(\phi) := \vdual{u_d}{\phi}_\Omega - \left( \lambda+1 \right) \vdual{u}{\phi}_\Omega
    - \lambda\vdual{\nabla u}{\nabla\phi}_\Omega
  \end{align*}
  From $u-\varphi_j\in H^1_{\rm ad}(\Omega)$ we conclude that
  $\vdual{\nabla\varphi_j}{\nabla\varphi_j}_\Omega \leq \ell(\varphi_j)$, and from
  $u+\varphi_j\in H^1_{\rm ad}(\Omega)$ we conclude that
  $2\vdual{\nabla\varphi_1}{\nabla \varphi_2}_\Omega\geq \ell(\varphi_1) + \ell(\varphi_2)$.
  Finally,
  \begin{align*}
    \vdual{\nabla(\varphi_1-\varphi_2)}{\nabla(\varphi_1-\varphi_2)}_\Omega
    = \vdual{\nabla\varphi_1}{\nabla\varphi_1}_\Omega
    + \vdual{\nabla\varphi_1}{\nabla\varphi_1}_\Omega
    - 2\vdual{\nabla\varphi_1}{\nabla\varphi_2}_\Omega
    \leq 0,
  \end{align*}
  and we conclude that $\varphi_1=\varphi_2$.
\end{proof}
\subsection{Regularity of the solution}
Problem~\eqref{eq:var:general} is actually a Neumann problem for $-\Delta u + u=g$ with
zero Neumann data. It is known that on smooth domains or convex polygons in $\R^2$,
the solution has improved regularity $u\in H^2(\Omega)$.
Although the space of test functions for~\eqref{eq:problem:var} is a proper subspace of $H^1(\Omega)$,
we can still obtain the same regularity result.
\begin{lemma}
  \label{lem:new:regularity}
  Let $\Omega$ be either smooth, i.e., having a $C^{1,1}$ boundary, or a convex polygon in $\R^2$.
  Suppose that $g\in L_2(\Omega)$ and
  let $u\in \cH^1(\Omega)$ be the unique solution of~\eqref{eq:problem:var}.
  Then $u\in H^{2}(\Omega)$ and
  \begin{align*}
    \norm{u}{H^2(\Omega)} \lesssim \norm{g}{L_2(\Omega)}.
  \end{align*}
\end{lemma}
\begin{proof}
  Let $u\in\cH^1(\Omega)$ be the solution of~\eqref{eq:problem:var}. For $v\in H^1(\Omega)$,
  write $v = \Ee v + (v - \Ee v)$ and note that $v-\Ee v\in H^1_0(\Omega)$ to conclude
  \begin{align}\label{lem:reg:eq1}
    (\lambda+1)\vdual{u}{v}_\Omega + \lambda\vdual{\nabla u}{\nabla v}_\Omega
    = \vdual{g}{\Ee v}_\Omega + (\lambda+1)\vdual{u}{v-\Ee v}_\Omega
    \quad\text{ for all } v \in H^1(\Omega).
  \end{align}
  Now, according to~\cite[Lemma~2.1]{ApelNP_16}, cf. also~\cite[Lemma~2.1]{MayRV_13},
  it holds that
  \begin{align}\label{eq:harm}
    \norm{\Ee v}{L_2(\Omega)} \lesssim \norm{v}{H^{-1/2}(\Gamma)},
  \end{align}
  and with~\eqref{lem:new:wellposed:eq1}
  we conclude that the right-hand side of~\eqref{lem:reg:eq1} is bounded by
  \begin{align*}
    \norm{g}{L_2(\Omega)} \left( \norm{v}{L_2(\Omega)} + \norm{v}{H^{-1/2}(\Gamma)}
    \right)
  \end{align*}
  Hence, the variational formulation~\eqref{lem:reg:eq1}
  for $u$ is actually a Neumann problem of the form
  \begin{align*}
    \lambda\vdual{\nabla u}{\nabla v}_\Omega
    = 
    (\lambda+1)\vdual{u}{v}_\Omega +
    \vdual{F}{v}_\Omega + \dual{h}{v}_\Gamma
    \quad\text{ for all } v \in H^1(\Omega)
  \end{align*}
  for some $F\in L_2(\Omega)$ and $h\in H^{1/2}(\Gamma)$ with norms bounded
  by $\norm{g}{L_2(\Omega)}$. According to~\cite[Theorem~1.10]{GiraultR_86} 
  there holds the regularity $u\in H^2(\Omega)$.
\end{proof}
\section{Analysis of the discrete problem}\label{sec:disc}
We assume for simplicity that $\Omega\subset\R^d$ is a polyhedral domain with boundary $\Gamma$. Let
$\Tt_h$ be a regular, $\kappa$-shape regular mesh on $\Omega$, i.e., $\Tt_h$ is a finite set of open,
$d$-dimensional simplices $K\in\Tt_h$ such that $\overline\Omega = \cup_{K\in\Tt_h}\overline K$,
there are no hanging nodes, and $h_K^d \simeq \abs{K}$, where $h_K := \diam(K)$ denotes
the element diameter and $\abs{K}$ denotes the volume of $K$. We denote the global
mesh-width by $h = \max_{K\in\Tt_h}h_K$.
We define the space of globally continuous,
piecewise polynomials of degree $p\in\N$ as
\begin{align*}
  \Ss^1(\Tt_h) :=
  \left\{ u_h \in C(\Omega) \mid u_h|_K \textit{ is a polynomial of degree at most } p \right\}.
\end{align*}
By $\Ss^1(\Tt_h)|_\Gamma$ we denote the induced piecewise polynomial space on the boundary $\Gamma$,
and $\Ss^1_0(\Tt_h):=\Ss^1(\Tt_h)\cap H^1_0(\Omega)$.
We will use the Scott-Zhang projection operator $J_h:H^1(\Omega)\rightarrow \Ss^1(\Tt_h)$
from~\cite{ScottZ_90}. It has the well-known properties
\begin{subequations}
\begin{align}
  p_h - J_h p_h &= 0 \quad\text{ for all }p_h\in\Ss^1(\Tt_h),\label{sz:eq1}\\
  v - J_h v &\in H^1_0(\Omega)  \quad\text{ if } v|_\Gamma = \Ss^1(\Tt_h)|_\Gamma,\label{sz:eq2}\\
  \norm{J_h v }{H^1(\Omega)} &\lesssim \norm{v}{H^1(\Omega)} \quad\text{ for all } v\in H^1(\Omega),
  \label{sz:eq3}\\
  \norm{v - J_h v}{H^1(\Omega)} &\lesssim h \norm{v}{H^2(\Omega)}
  \quad\text{ for all } v\in H^1_0(\Omega)\cap H^2(\Omega).\label{sz:eq4}
\end{align}
\end{subequations}
By $\Ee_h:\Ss^1(\Tt_h)|_\Gamma \rightarrow \Ss^1(\Tt_h)$ we denote the discrete harmonic extension
operator, i.e., the discretization of~\eqref{eq:harmext}, that is
\begin{align}\label{eq:discharmext}
  \begin{split}
    \vdual{\nabla\Ee_h q_h}{\nabla v_h}_\Omega &= 0 \quad\text{ for all } v_h\in \Ss^1_0(\Tt_h),\\
    \gamma\Ee_h q_h &= q_h.
  \end{split}
\end{align}
Note that
\begin{align}
  \trace \Ee_h \trace p_h = \trace p_h \quad\text{ for all } p_h\in\Ss^1(\Tt_h).
  \label{dext:eq}
\end{align}
Finally, we define the finite-dimensional space of discrete harmonic functions
\begin{align*}
  \widetilde\Ss^1(\Tt_h) := \Ee_h(\Ss^1(\Tt_h)|_\Gamma).
\end{align*}
Again, the spaces $\Ss^1_{\rm ad}(\Tt_h)$ and $\wilde\Ss^1_{\rm ad}(\Tt_h)$ denote convex subsets
of functions which fulfill the constraint~\eqref{eq:problem:new:box}.
\subsection{Nonconforming a priori analysis}
The following lemma follows straighforwardly as in the continuous case, cf. Corollary~\ref{cor:cont:equiv}.
\begin{lemma}\label{lem:nconf}
  The problem~\eqref{eq:var:disc} has a unique solution $(u_h,\varphi_h)\in\Ss^1(\Tt_h)\times \Ss^1_0(\Tt_h)$,
  and the problem~\eqref{eq:var:constrained:disc} has a unique solution
  $(u_h,\varphi_h)\in\Ss^1_{\rm ad}(\Tt_h)\times\Ss^1_0(\Tt_h)$.
  Furthermore, if $(u_h,\varphi_h)\in\Ss^1(\Tt_h)\times\Ss^1_0(\Tt_h)$ solves~\eqref{eq:var:disc}, then $u_h$
  solves the problem to
  \begin{align}\label{eq:problem:var:disc:u}
    \text{find } \wilde u_h \in \widetilde\Ss^1(\Tt_h) \text{ such that }
    b(\wilde u_h,\psi_h)= \vdual{u_d}{\psi_h}_\Omega
    \text{ for all } \psi_h\in \widetilde\Ss^1(\Tt_h).
  \end{align}
  Likewise, if $(u_h,\varphi_h)\in\Ss^1_{\rm ad}(\Tt_h)\times\Ss^1_0(\Tt_h)$ solves~\eqref{eq:var:constrained:disc}, then $u_h$
  solves the problem to
  \begin{align}\label{eq:problem:var:constrained:disc:u}
    \text{find } \wilde u_h \in \widetilde\Ss^1_{\rm ad}(\Tt_h) \text{ such that }
    b(\wilde u_h,\psi_h-\wilde u_h)\geq \vdual{u_d}{\psi_h-\wilde u_h}_\Omega
    \text{ for all } \psi_h\in \widetilde\Ss^1_{\rm ad}(\Tt_h).
  \end{align}
  Furthermore, solutions of~\eqref{eq:problem:var:disc:u} and~\eqref{eq:problem:var:constrained:disc:u}
  are unique.
\end{lemma}
For later reference, suppose that 
$\phi\in \cH^1(\Omega)$ is the solution of problem~\eqref{eq:problem:var}.
Then, we define a finite element approximation $\phi_h\in\wilde\Ss^1(\Tt_h)$ as
\begin{align}\label{eq:problem:var:disc}
  \text{find } \phi_h \in \widetilde\Ss^1(\Tt_h) \text{ such that } b(\phi_h,\psi_h)= \vdual{g}{\psi_h}_\Omega
  \text{ for all } \psi_h\in \widetilde\Ss^1(\Tt_h).
\end{align}
Indeed, as $\wilde\Ss^1(\Tt_h)$ is not a subset of $\cH^1(\Omega)$, the
finite element discretizations~\eqref{eq:problem:var:disc:u}-\eqref{eq:problem:var:disc}
are nonconforming, and the a priori analysis
for these problems has to be done accordingly.
First, we will show an $H^1(\Omega)$ bound for the error by applying the
Strang lemma. This leads to a best approximation term involving
discrete harmonic functions. In order to facilitate the application of polynomial approximation results,
we will use the following result, which tells us
that the best approximation of a harmonic function with discrete harmonic functions
is as good as the best approximation by piecewise polynomials.
\begin{lemma}\label{lem:cea:harmonic}
  Let $\phi\in \cH^1(\Omega)$. Then
  \begin{align*}
    \inf_{\phi_h\in \widetilde\Ss^1(\Tt_h)} \norm{\phi-\phi_h}{H^1(\Omega)}
    \lesssim
    \inf_{w_h\in \Ss^1(\Tt_h)} \norm{\phi-w_h}{H^1(\Omega)}
  \end{align*}
\end{lemma}
\begin{proof}
  Let $w_h\in\Ss^1(\Tt_h)$ be arbitrary. Then
  \begin{align*}
    \inf_{\phi_h\in\widetilde\Ss^1(\Tt_h)}\norm{\phi-\phi_h}{H^1(\Omega)} \leq
    \norm{\phi-\Ee_h\trace w_\Tt}{H^1(\Omega)}.
  \end{align*}
  According to the norm equivalence~\eqref{eq:H1norm} and the identity~\eqref{dext:eq}, we have
  \begin{align*}
    \norm{\phi-\Ee_h\trace w_h}{H^1(\Omega)} \simeq \norm{\nabla(\phi-\Ee_h\trace w_h)}{L_2(\Omega)}
    + \norm{\trace \phi-\trace w_h}{H^{1/2}(\Gamma)}.
  \end{align*}
  We have the identity
  \begin{align*}
    \vdual{\nabla (\phi-\Ee_h\trace w_h)}{\nabla (\phi-\Ee_h\trace w_h)}
    = \vdual{\nabla (\phi-\Ee_h\trace w_h)}{\nabla (\phi-J_h\Ee\trace w_h)},
  \end{align*}
  which follows from the fact that $\phi$ and $\Ee_h \gamma w_\Tt$ are both discrete harmonic
  and $\Ee_h\trace w_h - J_h\Ee\trace w_h\in \Ss^1_0(\Tt_h)$ due 
  to~\eqref{sz:eq2} and~\eqref{dext:eq}. Using Cauchy-Schwarz, this yields
  \begin{align*}
    \norm{\nabla (\phi-\Ee_h\gamma w_h)}{L_2(\Omega)}
    \leq \norm{\nabla (\phi-J_h\Ee\gamma w_h)}{L_2(\Omega)}
    &\leq \norm{\phi-J_h \phi}{H^1(\Omega)} + \norm{J_h\Ee\gamma (\phi- w_h)}{H^1(\Omega)}\\
    &\lesssim \norm{\phi-J_h \phi}{H^1(\Omega)} + \norm{\gamma(\phi-w_h)}{H^{1/2}(\Gamma)}\\
    &\lesssim \norm{(\phi-w_h)-J_h(\phi-w_h)}{H^1(\Omega)} + \norm{\gamma(\phi-w_h)}{H^{1/2}(\Gamma)}\\
    &\lesssim \norm{\phi-w_h}{H^1(\Omega)},
  \end{align*}
  where we have also used the triangle inequality, $\phi = \Ee\gamma \phi$,~\eqref{sz:eq1},
  \eqref{sz:eq3}, and again~\eqref{eq:H1norm}.
\end{proof}
We will also need the following simple application of the Aubin-Nitsche trick.
\begin{lemma}\label{lem:an:discrete}
  Suppose that $\psi_h\in\widetilde\Ss^1(\Tt_h)$. Then it holds
  \begin{align*}
    \norm{\Ee\psi_h - \psi_h}{L_2(\Omega)} \leq
    \norm{\nabla(\Ee\psi_h - \psi_h)}{L_2(\Omega)}
    \sup_{\norm{f}{L_2(\Omega)}=1} \inf_{v_h\in \Ss^1_0(\Tt_h)} \norm{v_f-v_h}{H^1(\Omega)},
  \end{align*}
  where $v_f\in H^1_0(\Omega)$ is the solution of $-\Delta v_f=f$ and $v_f = 0$ on $\Gamma$.
\end{lemma}
\begin{proof}
  Due to the fact that $\psi_h-\Ee\psi_h\in H^1_0(\Omega)$ and discrete harmonic, we 
  have for arbitrary $f\in L_2(\Omega)$ and $v_\Tt\in\Ss^1_0(\Tt_h)$ the identity
  \begin{align*}
    \vdual{\psi_h-\Ee\trace \psi_h}{f}_\Omega
    = \vdual{\nabla v_f}{\nabla (\psi_h-\Ee\trace \psi_h)}_\Omega
    = \vdual{\nabla (v_f - v_h)}{\nabla (\psi_h-\Ee\trace \psi_h)}_\Omega
  \end{align*}
  An application of Cauchy-Schwarz finishes the proof.
\end{proof}
Now we can show an $H^1(\Omega)$ bound for the error.
\begin{theorem}\label{thm:cea}
  With $\phi\in \cH^1(\Omega)$ the solution of~\eqref{eq:problem:var} and
  $\phi_h\in\widetilde\Ss^1(\Tt_h)$ the solution of~\eqref{eq:problem:var:disc}, there holds
  \begin{align}
    \norm{\phi-\phi_h}{H^1(\Omega)} \lesssim \inf_{w_h\in\Ss^1(\Tt_h)}\norm{\phi-w_h}{H^1(\Omega)}
    + \sup_{\norm{f}{L_2(\Omega)}=1} \inf_{v_h\in \Ss^1_0(\Tt_h)} \norm{v_f-v_h}{H^1(\Omega)},
    \label{thm:apriori:eq}
  \end{align}
  where $v_f\in H^1_0(\Omega)$ is the solution of $-\Delta v_f=f$ and $v_f = 0$ on $\Gamma$.
\end{theorem}
\begin{proof}
  For $w_h\in\widetilde\Ss^1(\Tt_h)$, a Strang-type argument shows
  \begin{align}\label{thm:cea:eq3}
    \norm{\phi_h-w_h}{H^1(\Omega)} &\lesssim \norm{\phi-w_h}{H^1(\Omega)} +
    \sup_{\psi_h\in\widetilde\Ss^1(\Tt)}
    \frac{\abs{b(\phi,\psi_h)-\vdual{g}{\psi_h}_\Omega}}{\norm{\psi_h}{H^1(\Omega)}}.
  \end{align}
  For $\psi_h\in\widetilde\Ss^1(\Tt_h)$, we have
  \begin{align*}
    b(\phi,\psi_h)-\vdual{g}{\psi_h}_\Omega
    &= b(\phi,\psi_h - \Ee\trace \psi_h) - \vdual{g}{\psi_h-\Ee\trace \psi_h}_\Omega\\
    &= (\lambda+1)\vdual{\phi}{\psi_h-\Ee\trace \psi_h}_\Omega
    - \vdual{g}{\psi_h-\Ee\trace \psi_h}_\Omega.
  \end{align*}
  Here, the first identity follows from~\eqref{eq:problem:var} and the second one
  follow from the fact that $\phi\in\cH^1(\Omega)$ and $\psi_h-\Ee\trace \psi_h\in H^1_0(\Omega)$.
  Hence, an application of Lemma~\ref{lem:an:discrete} and
  $\norm{\nabla(\psi_h-\Ee\trace \psi_h)}{L_2(\Omega)} \lesssim \norm{\psi_h}{H^1(\Omega)}$ shows
  \begin{align}\label{thm:cea:eq1}
    \begin{split}
      \abs{b(\phi,\psi_h)-\vdual{g}{\psi_h}_\Omega}
      &\leq \norm{g}{L_2(\Omega)}\norm{\psi_h - \Ee\psi_h}{L_2(\Omega)}\\
      &\lesssim 
      \norm{\psi_h}{H^1(\Omega)}
      \sup_{\norm{f}{L_2(\Omega)}=1} \inf_{v_h\in \Ss^1_0(\Tt_h)} \norm{v_f-v_h}{H^1(\Omega)}.
    \end{split}
  \end{align}
  Using the bound~\eqref{thm:cea:eq1} on the right-hand side of~\eqref{thm:cea:eq3}
  and the triangle inequality, this yields
  \begin{align*}
    \norm{\phi-\phi_h}{H^1(\Omega)} \lesssim \inf_{w_h\in\widetilde\Ss^1(\Tt_h)}\norm{\phi-w_h}{H^1(\Omega)}
    + \sup_{\norm{f}{L_2(\Omega)}=1} \inf_{v_h\in \Ss^1_0(\Tt_h)} \norm{v_f-v_h}{H^1(\Omega)}.
  \end{align*}
  An application of Lemma~\ref{lem:cea:harmonic} finishes the proof.
\end{proof}
Now we can prove an $L_2(\Omega)$ bound for the error.
\begin{theorem}\label{thm:an}
  Let $u\in\cH^1(\Omega)$ be the solution of problem~\eqref{eq:problem:new} and
  $u_h\in\wilde\Ss^1(\Tt_h)$ its finite element approximation. Then it holds
  \begin{align*}
    \norm{u - u_h}{L_2(\Omega)} \lesssim &\norm{u-u_h}{H^1(\Omega)}
    \Bigl( \sup_{\norm{g}{L_2(\Omega)}=1} \inf_{w_h\in\Ss^1(\Tt_h)}\norm{\phi_g-w_h}{H^1(\Omega)}
    +\sup_{\norm{f}{L_2(\Omega)}=1} \inf_{v_h\in \Ss^1_0(\Tt_h)} \norm{v_f-v_h}{H^1(\Omega)}
    \Bigr)\\
    &+
    \sup_{\norm{g}{L_2(\Omega)}=1} \inf_{w_g\in\Ss^1(\Tt_h)}\norm{\phi_g-w_h}{H^1(\Omega)}
    \sup_{\norm{f}{L_2(\Omega)}=1} \inf_{v_h\in \Ss^1_0(\Tt_h)} \norm{v_f-v_h}{H^1(\Omega)}\\
    &+ \Bigl( 
    \sup_{\norm{f}{L_2(\Omega)}=1} \inf_{v_h\in \Ss^1_0(\Tt_h)} \norm{v_f-v_h}{H^1(\Omega)}
    \Bigr)^2,
  \end{align*}
  where $v_f\in H^1_0(\Omega)$ is the solution of $-\Delta v_f=f$ and $v_f = 0$ on $\Gamma$,
  and $\phi_g\in \cH^1(\Omega)$ is the solution of~\eqref{eq:problem:var}.
\end{theorem}
\begin{proof}
  The Aubin-Nitsche trick for nonconforming methods, cf.~\cite[Chapter~III, Lemma~1.4]{braess}
  states
  \begin{align*}
    \norm{u-u_h}{L_2(\Omega)}
    \leq \sup_{g\in L_2(\Omega)} \frac{1}{\norm{g}{L_2(\Omega)}}\Bigl\{
      &C \norm{u-u_h}{H^1(\Omega)}\norm{\phi_g - \phi_h}{H^1(\Omega)}\\
      &+ \abs{b(u-u_h,\phi_g) - \vdual{g}{u-u_h}}\\
      &+ \abs{b(u,\phi_g-\phi_h) - \vdual{u_d}{\phi_g-\phi_h}}
    \Bigr\},
  \end{align*}
  where $\phi_g\in \cH^1(\Omega)$ denotes the solution of~\eqref{eq:problem:var} and
  $\phi_h\in\wilde\Ss^1(\Tt_h)$ its nonconforming finite element approximation~\eqref{eq:problem:var:disc}.
  As $\phi_g$ solves~\eqref{eq:problem:var}, we conclude
  \begin{align*}
    b(u-u_h,\phi_g) - \vdual{g}{u-u_h}_\Omega
    &= b(\Ee u_h-u_h,\phi_g) - \vdual{g}{\Ee u_h-u_h}_\Omega\\
    &= (\lambda+1)\vdual{\Ee u_h-u_h}{\phi_g}_\Omega - \vdual{g}{\Ee u_h-u_h}_\Omega,
  \end{align*}
  where in the last step we used $\phi_g\in \cH^1(\Omega)$ and $\Ee u_h-u_h\in H^1_0(\Omega)$.
  Cauchy-Schwarz and~\eqref{lem:new:wellposed:eq1} show
  \begin{align*}
    \abs{b(u-u_h,\phi_g) - \vdual{g}{u-u_h}_\Omega}
    &\leq \norm{g}{L_2(\Omega)}\norm{\Ee u_h-u_h}{L_2(\Omega)}\\
    &\lesssim \norm{g}{L_2(\Omega)}
    \norm{\nabla(u - u_h)}{L_2(\Omega)} \sup_{\norm{f}{L_2(\Omega)}=1} \inf_{v_h\in \Ss^1_0(\Tt_h)} \norm{v_f-v_h}{H^1(\Omega)},
  \end{align*}
  where we have used Lemmas~\ref{lem:an:discrete} and~\eqref{eq:harm:best} in the last step.
  The same argument can be applied to the second term, i.e.,
  \begin{align*}
    b(u,\phi_g-\phi_h) - \vdual{u_d}{\phi_g-\phi_h}
    &= b(u,\Ee\phi_h-\phi_h) - \vdual{u_d}{\Ee\phi_h-\phi_h}\\
    &= (\lambda+1)\vdual{u}{\Ee \phi_h-\phi_h}_\Omega - \vdual{u_d}{\Ee\phi_h-\phi_h}_\Omega,
  \end{align*}
  and again Lemmas~\ref{lem:an:discrete} and~\eqref{eq:harm:best} show
  \begin{align*}
    \abs{b(u,\phi_g-\phi_h) - \vdual{u_d}{\phi_g-\phi_h}}
    &\lesssim \norm{u_d}{L_2(\Omega)}\norm{\Ee\phi_h-\phi_h}{L_2(\Omega)}\\
    &\lesssim \norm{u_d}{L_2(\Omega)}\norm{\nabla(\phi_g-\phi_h)}{L_2(\Omega)}
    \sup_{\norm{f}{L_2(\Omega)}=1} \inf_{v_h\in \Ss^1_0(\Tt_h)} \norm{v_f-v_h}{H^1(\Omega)}
  \end{align*}
  Applying Theorem~\ref{thm:cea} to $\norm{\phi_g-\phi_h}{H^1(\Omega)}$
  and arranging the resulting terms proves the statement.
\end{proof}
We can now prove Theorem~\ref{thm:est}.
\begin{proof}[Proof of Theorem~\ref{thm:est}]
  According to Lemma~\ref{lem:new:regularity}, we have the regularity $u\in H^2(\Omega)$.
  Hence, ubiquitous approximation results in Sobolev spaces prove
  \begin{align*}
    \inf_{w_h\in\Ss^1(\Tt_h)}\norm{u - w_h}{H^1(\Omega)} = \OO(h).
  \end{align*}
  If $v_f\in H^1_0(\Omega)$ is the solution of $-\Delta v_f=f$ and $v_f = 0$ on $\Gamma$,
  elliptic regularity results state that $v_f\in H^2(\Omega)$
  and $\norm{v_f}{H^2(\Omega)}\lesssim\norm{f}{L_2(\Omega)}$.
  Again, approximation results in Sobolev spaces yield
  \begin{align*}
    \sup_{\norm{f}{L_2(\Omega)}=1} \inf_{v_h\in \Ss^1_0(\Tt_h)} \norm{v_f-v_h}{H^1(\Omega)} = \OO(h).
  \end{align*}
  Theorem~\ref{thm:cea} then shows $\norm{u-u_h}{H^1(\Omega)}=\OO(h)$.
  The statement $\norm{u-u_h}{H^1(\Omega)}=\OO(h^2)$ is shown analogously
  with Theorem~\ref{thm:an}.
  The estimate $\norm{u-u_h}{H^{1/2}(\Gamma)}=\OO(h)$ for the trace of the error
  then follows by the standard trace inequality. The multiplicative trace inequality
  \begin{align*}
    \norm{v}{L_2(\Gamma)} \lesssim \norm{v}{L_2(\Omega)}^{1/2}\norm{v}{H^1(\Omega)}^{1/2},
  \end{align*}
  cf.~\cite[Theorem 1.6.6]{BrennerS_08}, shows the error estimate in $L_2(\Gamma)$.
\end{proof}
\begin{proof}[Proof of Theorem~\ref{thm:cea:constrained}]
  We will follow ideas from~\cite{Falk_74}. Due to the ellipticity and symmetry of $b$
  (we denote the ellipticity constant by $\alpha_b$),
  we have for arbitrary $v\in\cH^1(\Omega)$ and $w_h\in\wilde\Ss^1_{\rm ad}(\Tt_h)$
  \begin{align}\label{thm:est:eq1}
    \begin{split}
      \alpha_b\norm{u-u_h}{H^1(\Omega)}^2 &\leq b(u-u_h,u-u_h)\\
      &\leq \vdual{u_d}{u_h-v}_\Omega - b(u,u_h-v) + \vdual{u_d}{u-w_h}_\Omega - b(u_h,u-w_h)\\
      &= \vdual{u_d}{u_h-v}_\Omega - b(u,u_h-v) + \vdual{u_d}{u-w_h}_\Omega - b(u,u-w_h)
      + b(u-u_h,u-w_h).
    \end{split}
  \end{align}
  Consider the first two terms on the right-hand side of~\eqref{thm:est:eq1}.
  We set $v=\Ee u_h$ and conclude
  \begin{align*}
    \abs{\vdual{u_d}{u_h-v}_\Omega - b(u,u_h-v)}
    &= \abs{\vdual{u_d}{u_h-\Ee u_h}_\Omega - (\lambda+1)\vdual{u}{u_h-\Ee u_h}_\Omega}\\
    &\leq \norm{u_d - (\lambda+1)u}{L_2(\Omega)}\norm{u_h-\Ee u_h}{L_2(\Omega)}\\
    &\leq C \norm{\nabla(\Ee u_h - u_h)}{L_2(\Omega)}\cdot 
    \sup_{\norm{f}{L_2(\Omega)}=1} \inf_{v_h\in \Ss^1_0(\Tt_h)} \norm{v_f-v_h}{H^1(\Omega)}\\
    &\leq \frac{\alpha_b}{4}\norm{u - u_h}{H^1(\Omega)}^2 +
    \frac{C^2}{\alpha_b} 
    \Bigl(
    \sup_{\norm{f}{L_2(\Omega)}=1} \inf_{v_h\in \Ss^1_0(\Tt_h)} \norm{v_f-v_h}{H^1(\Omega)}
    \Bigr)^2
  \end{align*}
  where we used Lemma~\ref{lem:an:discrete},
  property~\eqref{eq:harm:best}, and Youngs inequality in the last step.
  For the third and fourth term on the right-hand side of~\eqref{thm:est:eq1}, we
  integrate by parts to see
  \begin{align*}
    \abs{\vdual{u_d}{u-w_h}_\Omega - b(u,u-w_h)}
    &= \abs{\vdual{u_d-(\lambda+1)u}{u-w_h} + \lambda \vdual{\Delta u}{u-w_h}_\Omega
    + \lambda \dual{\partial_n u}{u-w_h}_\Gamma}\\
    &\leq \left( \norm{u_d-(\lambda+1)u}{L_2(\Omega)} + \lambda \norm{u}{H^2(\Omega)} \right)
    \norm{u-w_h}{L_2(\Omega)}\\
    &\qquad + \lambda \norm{\partial_n u}{H^{1/2-\eps}(\Gamma)}\norm{u-w_h}{H^{-1/2+\eps}(\Gamma)}.
  \end{align*}
  The triangle inequality, Lemma~\ref{lem:an:discrete}, and the estimate~\eqref{eq:harm:best} show
  \begin{align*}
    \norm{u-w_h}{L_2(\Omega)} &\lesssim
    \norm{u-\Ee w_h}{L_2(\Omega)}+
    \norm{\nabla(u-w_h)}{L_2(\Omega)} \cdot
    \sup_{\norm{f}{L_2(\Omega)}=1} \inf_{v_h\in \Ss^1_0(\Tt_h)} \norm{v_f-v_h}{H^1(\Omega)}\\
    &\lesssim
    \norm{u-w_h}{H^{-1/2}(\Gamma)} + \norm{\nabla(u-w_h)}{L_2(\Omega)}^2+
    \Bigl(
    \sup_{\norm{f}{L_2(\Omega)}=1} \inf_{v_h\in \Ss^1_0(\Tt_h)} \norm{v_f-v_h}{H^1(\Omega)}.
    \Bigr)^2,
  \end{align*}
  where we used Youngs inequality and the estimate~\eqref{eq:harm} in the last step.
  The last term on the right-hand side of~\eqref{thm:est:eq1} is bounded by
  \begin{align*}
    b(u-u_h,u-w_h)\leq \frac{\alpha_b}{4} \norm{u-u_h}{H^1(\Omega)}^2 + \frac{C_b^2}{\alpha_b}\norm{u-w_h}{H^1(\Omega)}^2
  \end{align*}
  using the boundedness of $b$, indicated by the constant $C_b$, and Youngs inequality.
  Plugging together the preceding estimates, we arrive at
  \begin{align*}
    \frac{\alpha_b}{2} \norm{u-u_h}{H^1(\Omega)}^2 \lesssim
    \norm{u-w_h}{H^1(\Omega)}^2
    + \norm{u-w_h}{H^{-1/2+\eps}(\Gamma)} +
    \Bigl(
      \sup_{\norm{f}{L_2(\Omega)}=1} \inf_{v_h\in \Ss^1_0(\Tt_h)} \norm{v_f-v_h}{H^1(\Omega)}
    \Bigr)^2
  \end{align*}
  where $w_h\in\wilde\Ss^1_{\rm ad}(\Tt_h)$ is arbitrary. This proves the statement.
\end{proof}

\section{Numerical experiments}
In this section, we will present numerical results in two space dimensions and for $p=1$
to validate our theoretical findings.
We implemented the numerical experiments in Matlab and based our code
on the package \texttt{p1afem}, cf.~\cite{FunkenWP_11}.
\subsection{Experiment 1}
We choose $\Omega = (0,1)\times(0,1)$, $f=0$ and $u_d(x,y) = x(1-y)-0.35$ and do not impose
constraints.
We use a sequence of uniformly refined triangular meshes $\Tt_j$ and denote by $N_j$ the number of triangles
of the mesh $\Tt_j$.
As $\Omega$ is a convex polygonal, we expect the convergence rates as stipulated in Corollary~\ref{cor:est},
that is
\begin{align*}
  \norm{u-u_j}{H^1(\Omega)}^2 = \OO(N_j^{-1}),
  \quad\text{ and }\quad \norm{u-u_j}{L_2(\Omega)}^2 = \OO(N_j^{-2}).
\end{align*}
As the exact solution $u$ is unknown, we use the finest approximate solution $u_J$ as
reference solution and plot the squared errors $\norm{u_J-u_j}{H^1(\Omega)}^2$ and
$\norm{u_J-u_j}{L_2(\Omega)}^2$ for $j=1,\dots, J-1$. In Figure~\ref{fig:1}, we plot these errors
and observe the theoretically predicted rates. In Figure~\ref{fig:1}, we plot the approximate solution on a
mesh with 65536 elements for comparison with the solution of Experiment 2 in Figure~\ref{fig:4}.
\begin{figure}
  \centering
  \includegraphics[width=0.8\textwidth]{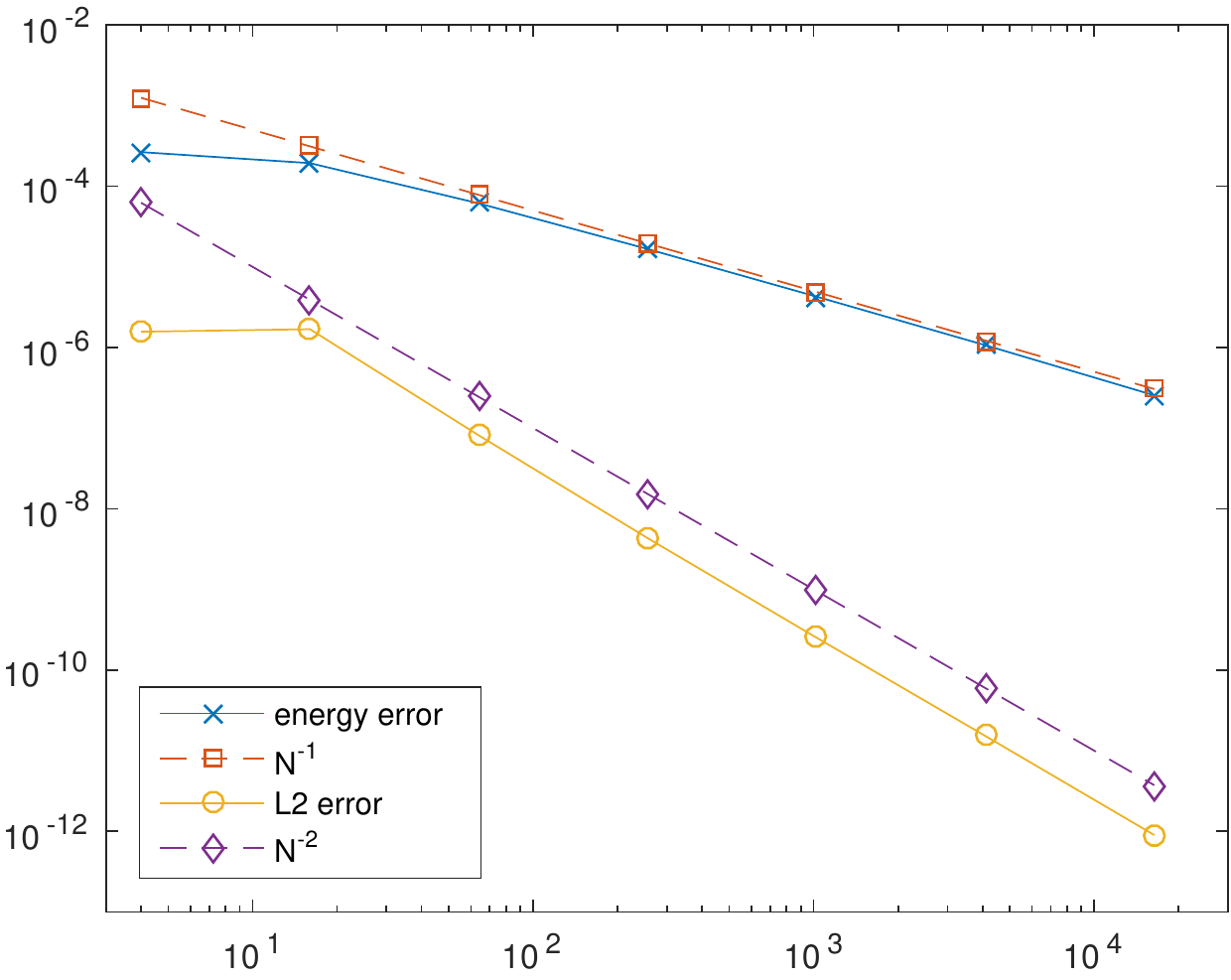}
  \caption{Squared $H^1(\Omega)$ and $L_2(\Omega)$ errors from Experiment 1.}
  \label{fig:1}
\end{figure}
\begin{figure}[htb]
  \centering
  \includegraphics[width=0.8\textwidth]{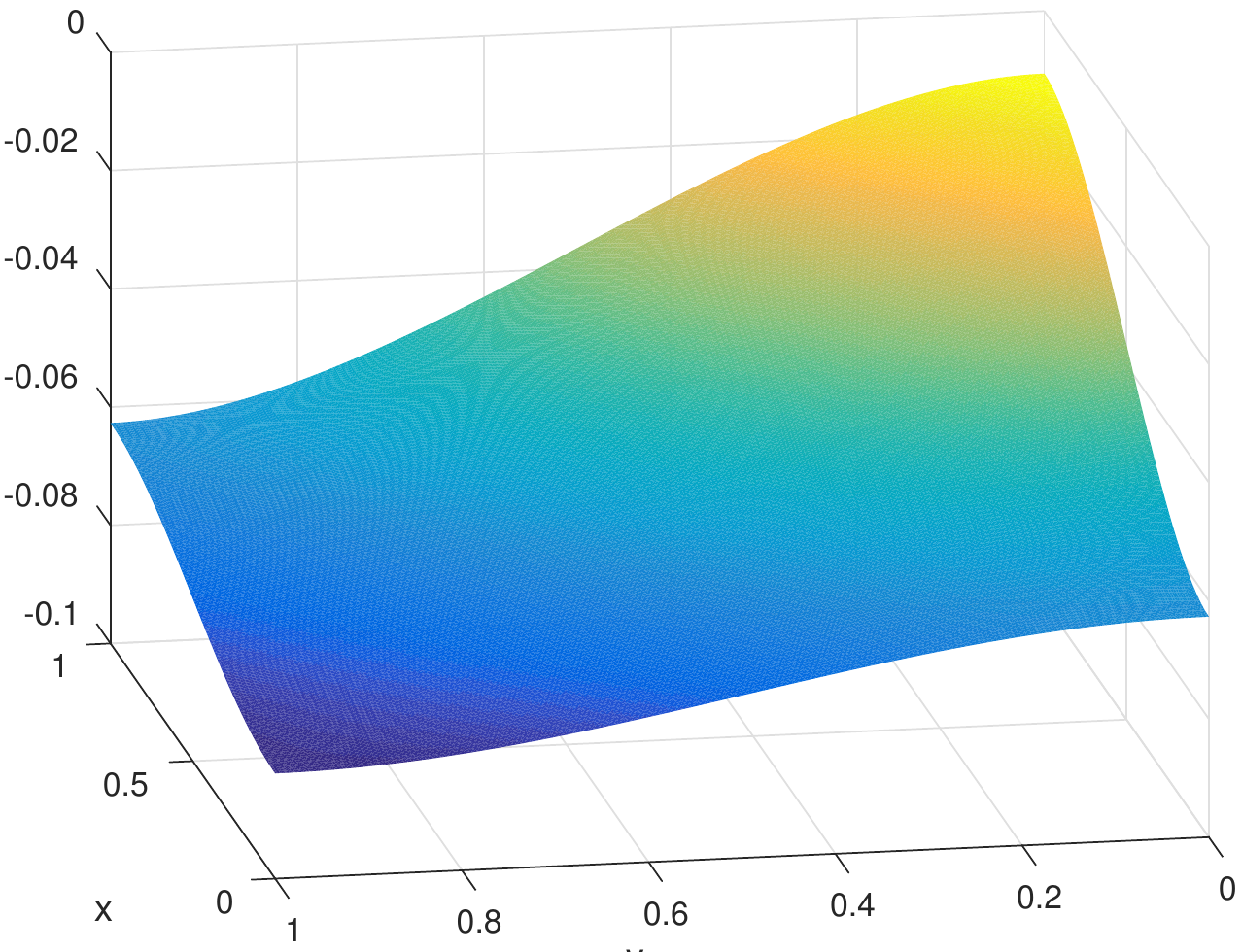}
  \caption{Approximate solution from Experiment 1.}
  \label{fig:2}
\end{figure}
\subsection{Experiment 2}
We choose the same data as in Experiment 1 and impose the restriction $0\leq g$ a.e. on $\Gamma$.
We solve the discrete system~\eqref{eq:var:constrained:disc} with the \textit{Primal-Dual Active Set Algorithm}
A.1 from~\cite{KarkkainenKT_Obstacle_03}, cf. also~\cite{HoppeK_Obstacle_94}.
In Figure~\ref{fig:3} we plot the convergence rates in $H^1(\Omega)$ and $L_2(\Omega)$, using
the same method as in Experiment 1. Again we observe the optimal rates.
In Figure~\ref{fig:4} we plot the approximate solution. Note that due to the restriction $0\leq g$
and the fact that $u$ is harmonic, $u$ has to be positive due to the maximum principle.
\begin{figure}[htb]
  \centering
  \includegraphics[width=0.8\textwidth]{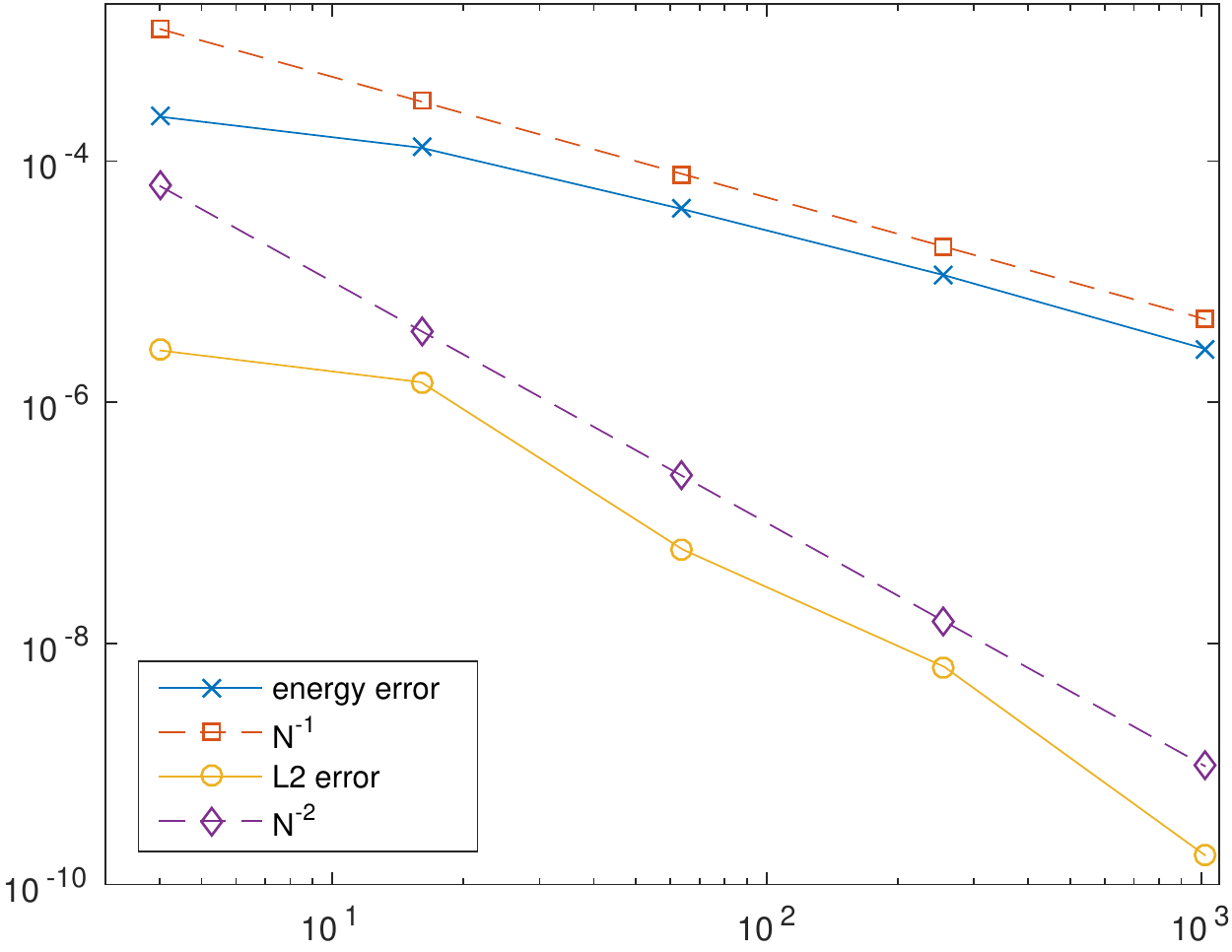}
  \caption{Squared $H^1(\Omega)$ and $L_2(\Omega)$ errors from Experiment 2.}
  \label{fig:3}
\end{figure}
\begin{figure}[htb]
  \centering
  \includegraphics[width=0.8\textwidth]{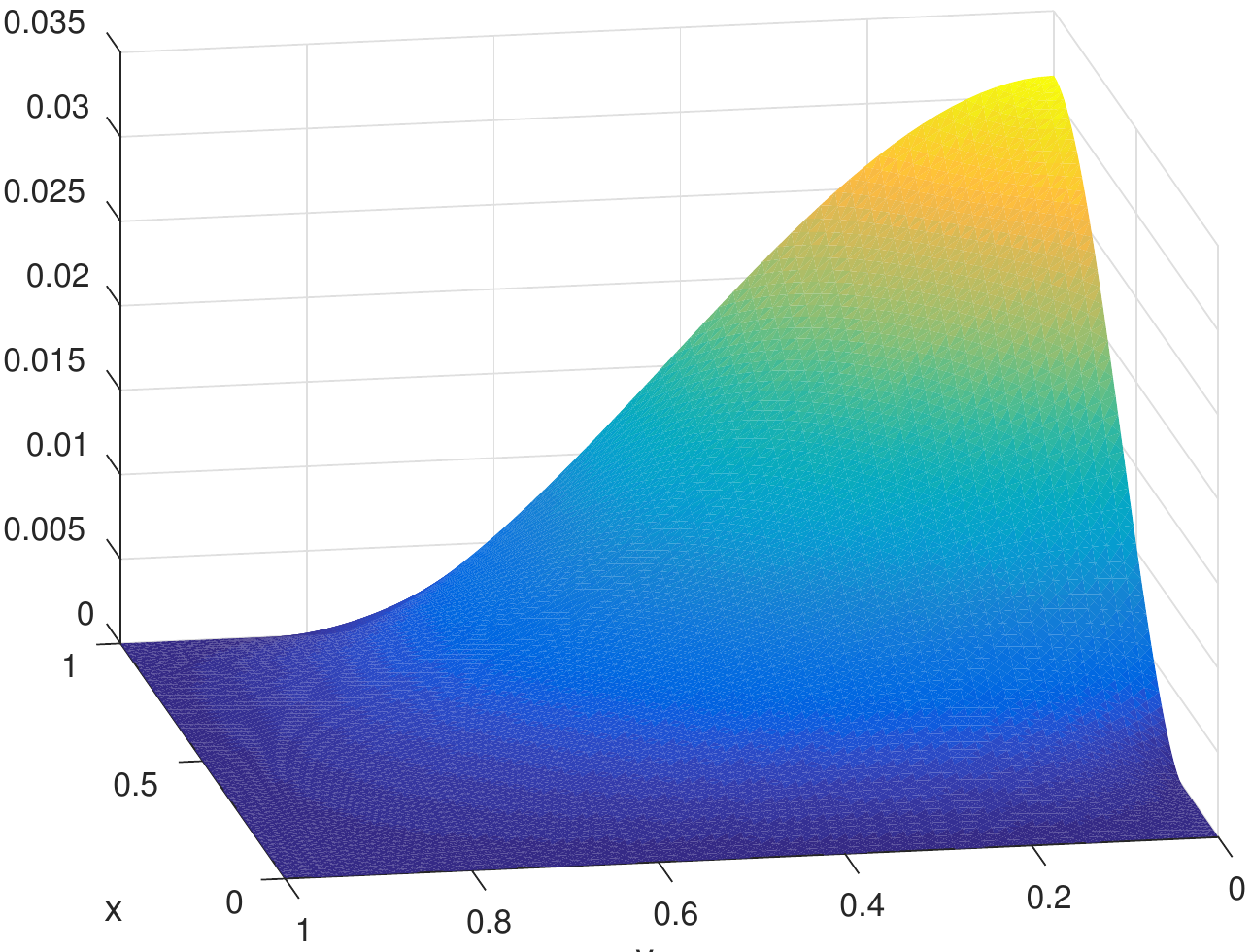}
  \caption{Approximate solution from Experiment 2.}
  \label{fig:4}
\end{figure}
\subsection{Experiment 3}
In this experiment, we use a problem with known exact solution, taken from~\cite{GongLTY_18}.
If we set
\begin{align*}
  u_d(x,y) &= 2\pi^2 \left( \cos(2\pi x)\sin^2(\pi y) + \sin^2(\pi x)\cos(2\pi y) \right),
\end{align*}
then the function $u(x,y)=0$ is the exact solution of the problem~\eqref{eq:problem:new:min}-\eqref{eq:problem:new:pde}.
We plot the squared $L_2(\Omega)$ and $H^1(\Omega)$ errors in Figure~\ref{fig:5}, and for both we
observe an error of $\OO(N^{-2})$. To explain this behaviour, note that as $u$ is part of our approximation space,
our method theoretically computes the exact solution. However,
the right-hand side in~\eqref{eq:var:disc} involves numerical quadrature of $u_d$, and so the discretization error
is bounded by the quadrature error.
\begin{figure}[htb]
  \centering
  \includegraphics[width=0.8\textwidth]{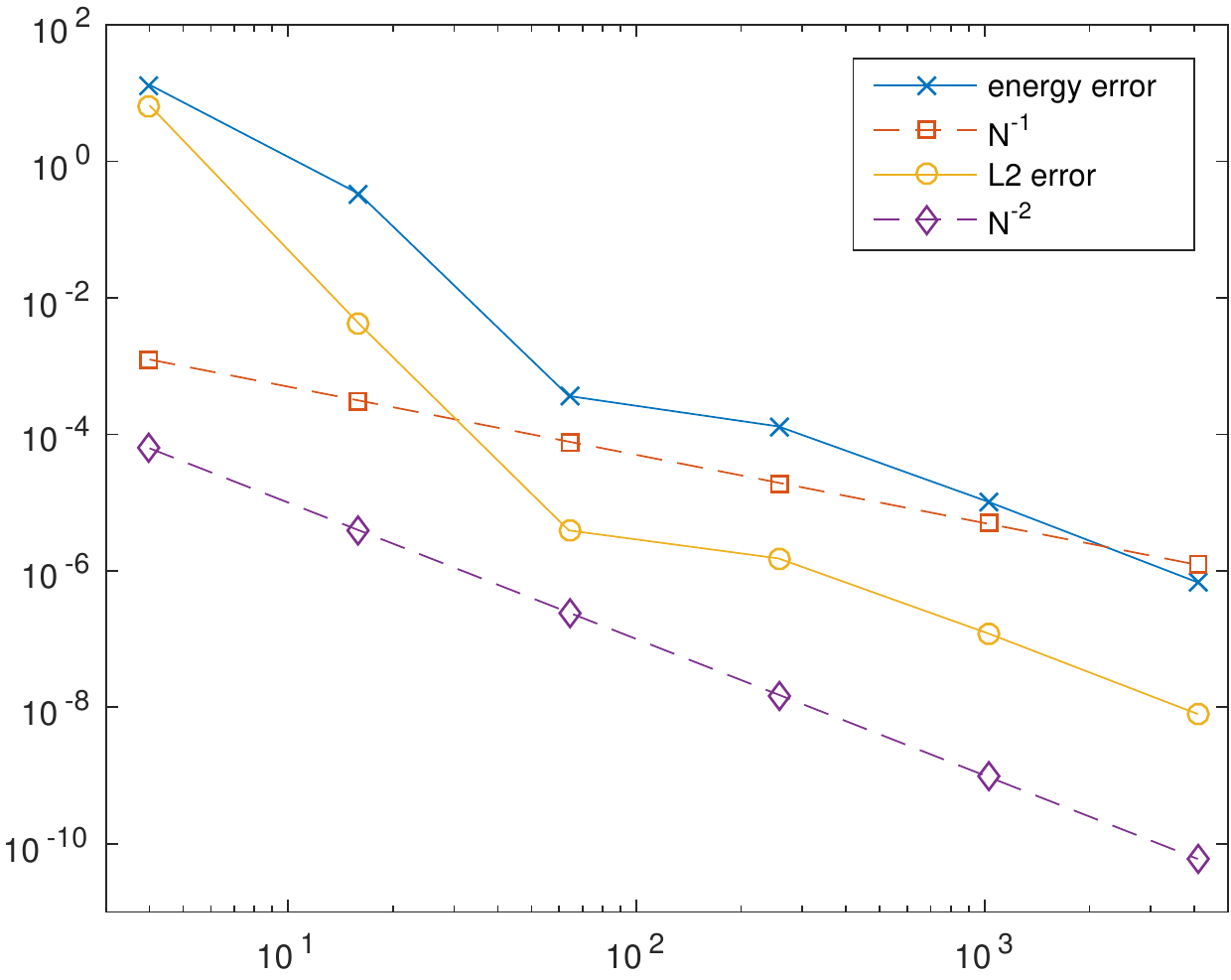}
  \caption{Squared $H^1(\Omega)$ and $L_2(\Omega)$ errors from Experiment 3.}
  \label{fig:5}
\end{figure}
\bigskip

\textbf{Acknowledgments:} The author would like to thank his colleagues Alejandro Allendes and Enrique Ot\'arola for fruitful discussions.
\bibliographystyle{abbrv}
\bibliography{literature}
\end{document}